\documentclass[a4paper,10pt]{amsart}
\usepackage{amsmath,amssymb,color,multirow}
\usepackage[colorlinks]{hyperref}
\definecolor{linkblue}{RGB}{1,1,190}
\definecolor{citered}{RGB}{190,1,1}
\hypersetup{linkcolor=linkblue,urlcolor=linkblue,citecolor=citered}

\theoremstyle{plain}
\newtheorem{theorem}{\bf Theorem}[section]
\newtheorem{proposition}[theorem]{\bf Proposition}
\newtheorem{lemma}[theorem]{\bf Lemma}

\theoremstyle{definition}
\newtheorem{example}[theorem]{\bf Example}
\newtheorem{definition}[theorem]{\bf Definition}

\newtheorem{conjecture}[theorem]{\bf Conjecture}

\hypersetup{hypertexnames=false}
\numberwithin{equation}{section}
\makeatletter\@namedef{subjclassname@2020}{\textup{2020} Mathematics Subject Classification}\makeatother

\begin{document}
\title[A counterexample]{A counterexample to the Pellian equation conjecture of Mordell}
\author{Andreas Reinhart}
\address{Institut f\"ur Mathematik und Wissenschaftliches Rechnen, Karl-Franzens-Universit\"at Graz, NAWI Graz, Heinrichstra{\ss}e 36, 8010 Graz, Austria}
\email{andreas.reinhart@uni-graz.at}
\thanks{This work was supported by the Austrian Science Fund FWF, Project Number P36852-N}

\begin{abstract}
Let $d\geq 2$ be a squarefree integer, let $\omega\in\{\sqrt{d},\frac{1+\sqrt{d}}{2}\}$ be such that $\mathbb{Z}[\omega]$ is the ring of algebraic integers of the real quadratic number field $\mathbb{Q}(\sqrt{d})$, let $\varepsilon>1$ be the fundamental unit of $\mathbb{Z}[\omega]$ and let $x$ and $y$ be the unique nonnegative integers with $\varepsilon=x+y\omega$. In this note, we extend and study the list of known squarefree integers $d\geq 2$, for which $y$ is divisible by $d$ (cf. OEIS A135735). As a byproduct, we present a counterexample to a conjecture of L. J. Mordell.
\end{abstract}

\maketitle

\section{Introduction, conjectures and terminology}\label{1}

Let $\mathbb{P}$, $\mathbb{N}$, $\mathbb{N}_0$, $\mathbb{Z}$, $\mathbb{Q}$ denote the sets of prime numbers, positive integers, nonnegative integers, integers and rational numbers, respectively. Let $f\in\mathbb{N}$. We say that $f$ is {\it squarefree} if $p^2\nmid f$ for each $p\in\mathbb{P}$. Moreover, $f$ is called {\it powerful} (also called {\it squareful}\,) if for each $p\in\mathbb{P}$ with $p\mid f$, we have $p^2\mid f$. Observe that $f$ is powerful if and only if $f=a^2b^3$ for some $a,b\in\mathbb{N}$.

\smallskip
Let $d\in\mathbb{N}_{\geq 2}$ be squarefree, let $K=\mathbb{Q}(\sqrt{d})$ and let $\mathcal{O}_K$ be the ring of algebraic integers of $K$. We set

\smallskip
\[
\omega=\begin{cases}\sqrt{d}&\textnormal{if }d\equiv 2,3\mod 4,\\\frac{1+\sqrt{d}}{2}&\textnormal{if }d\equiv 1\mod 4,\end{cases}\quad\quad\mathsf{d}_K=\begin{cases}
4d&\textnormal{if }d\equiv 2,3\mod 4,\\ d&\textnormal{if }d\equiv 1\mod 4.\end{cases}
\]

\bigskip
It is well-known that $\mathcal{O}_K=\mathbb{Z}[\omega]=\mathbb{Z}\oplus\omega\mathbb{Z}$. Let $\varepsilon\in\mathcal{O}_K$ be the (unique) fundamental unit with $\varepsilon>1$ (i.e., $\{\pm\varepsilon^k:k\in\mathbb{Z}\}$ is the unit group of $\mathcal{O}_K$). Observe that there always exist unique $x,y\in\mathbb{N}_0$ such that $\varepsilon=x+y\omega$, and if $d\not=5$, then $x,y\in\mathbb{N}$. From now on, we will use $x$ and $y$ as defined here.

\smallskip
So far, there are $17$ known squarefree integers $d\in\mathbb{N}_{\geq 2}$ with $d\mid y$ (see \cite[Remark 5.5]{Re23} or OEIS A135735). In this note, we extend the list of known squarefree integers $d\in\mathbb{N}_{\geq 2}$ with $d\mid y$ to $21$ members in total. One of the newly found numbers happens to be a counterexample to the Pellian equation conjecture of Mordell. For the readers' convenience, we include the complete list here:

\begin{align}
&46,430,1817,58254,209991,1752299,3124318,4099215,5374184665,6459560882,16466394154,\label{List 1.1}\tag{L1}\\
&20565608894,25666082990,117477414815,125854178626,1004569189366,1188580642033,\nonumber\\
&15826129757609,18803675974841,20256129307923,39028039587479\nonumber
\end{align}

\bigskip
For more details, we refer to Section~\ref{3} of this note (which contains several tables that summarize the properties of these numbers).

\smallskip
Next, we want to discuss the importance of the squarefree integers $d\in\mathbb{N}_{\geq 2}$ for which $d\mid y$. Indeed, there are several conjectures and results that are tied to these numbers. In what follows, we present these conjectures and results and restate them by using the aforementioned terminology.

\begin{conjecture}[The conjecture of Ankeny, Artin and Chowla or (AAC)-conjecture]\label{Conjecture 1.1} If $d\in\mathbb{P}$ and $d\equiv 1\mod 4$, then $d\nmid y$.
\end{conjecture}

The (AAC)-conjecture was first mentioned by N. C. Ankeny, E. Artin and S. Chowla in 1952 (see \cite[p. 480]{An-Ar-Ch52}) and has subsequently been studied by various authors. For instance, L. J. Mordell provided a characterization of the conjecture (if $d\equiv 5\mod 8$) that involves Bernoulli numbers \cite{Mo60,Mo61}. Also note that the conjecture plays some role in the study of direct-sum cancellation for modules over orders in real quadratic number fields \cite{Ha07}. For more recent work involving the (AAC)-conjecture, we refer to \cite{Be-Mo24,Si-Sh24}. Before \cite{Re23} was published, the conjecture has been verified (for all primes $d\equiv 1\mod 4$) up to $2\cdot 10^{11}$ (see \cite{Va-Te-Wi01,Va-Te-Wi03}). In \cite{Re23} the conjecture has been verified up to $1.5\cdot 10^{12}$ (but this was not stated explicitly).

\begin{conjecture}[The Pellian equation conjecture of Mordell]\label{Conjecture 1.2} If $d\in\mathbb{P}$ and $d\equiv 3\mod 4$, then $d\nmid y$.
\end{conjecture}

This conjecture was first formulated by A. A. Kiselev and I. Sh. Slavutski\u{\i} in 1959 (see \cite{Ki-Sl59}) and stated independently by L. J. Mordell in 1961 (see \cite[p. 283]{Mo61}) who also established a connection of this conjecture with Euler numbers. The conjecture of Mordell has recently been studied in a series of papers \cite{Be-Mo24,Ch-Sa19,Si-Sh24} and has been verified (for all primes $d\equiv 3\mod 4$) up to $1.6\cdot 10^9$ in \cite{Be-Mo24}. Around the same time, the Mordell conjecture has (independently) been verified up to $1.5\cdot 10^{12}$ in \cite{Re23}. In Section~\ref{3} we provide a counterexample to this conjecture (Example~\ref{Example 3.1}).

\begin{conjecture}[The conjecture of Erd\"os, Mollin and Walsh or (EMW)-conjecture]\label{Conjecture 1.3} For each $a\in\mathbb{N}$, there is some $b\in\{a,a+1,a+2\}$ such that $b$ is not powerful (i.e., there are no three consecutive powerful numbers).
\end{conjecture}

The (EMW)-conjecture was first mentioned in a paper of P. Erd\"os \cite{Er75} and has subsequently been rediscovered by R. A. Mollin and P. G. Walsh \cite{Mo-Wa86} who also provided a characterization of the conjecture in terms of fundamental units \cite{Mo87,Mo-Wa86}. This conjecture has wide implications (if it is true), like the existence of infinitely many primes that are not Wieferich primes \cite{Gr86}.

\medskip
Now we want to discuss various results that involve the squarefree integers $d\in\mathbb{N}_{\geq 2}$ with $d\mid y$. To do so, we need some more terminology. For $s,r,t\in\mathbb{N}_0$, let $[r,s]=\{z\in\mathbb{N}_0:r\leq z\leq s\}$ and $\mathbb{N}_{\geq t}=\{z\in\mathbb{N}_0:z\geq t\}$. Let ${\rm N}:K\rightarrow\mathbb{Q}$ defined by ${\rm N}(a+b\sqrt{d})=a^2-db^2$ for all $a,b\in\mathbb{Q}$ be the norm map on $K$. We call a subring $\mathcal{O}$ of $K$ with quotient field $K$ an {\it order} in $K$ if it is a finitely generated $\mathbb{Z}$-module. For each $f\in\mathbb{N}$, let $\mathcal{O}_f=\mathbb{Z}+f\mathcal{O}_K$ and note that $\mathcal{O}_f$ is the unique order in $K$ with conductor $f$ (i.e., $\{z\in\mathcal{O}_f:z\mathcal{O}_K\subseteq\mathcal{O}_f\}=f\mathcal{O}_K$). Let ${\rm Pic}(\mathcal{O})$ be the Picard group of $\mathcal{O}$ for each order $\mathcal{O}$ in $K$. We let ${\rm h}(d)=|{\rm Pic}(\mathcal{O}_K)|$ denote the class number of $K$. For all $a,b\in\mathbb{Z}$, let $\pmb{\Big(}\frac{a}{b}\pmb{\Big)}\in\{-1,0,1\}$ denote the Kronecker symbol of $a$ modulo $b$. If $p\in\mathbb{P}$, then $p$ is called {\it inert}, {\it ramified}, {\it split} (in $\mathcal{O}_K$) if $\pmb{\Big(}\frac{\mathsf{d}_K}{p}\pmb{\Big)}=-1$, $\pmb{\Big(}\frac{\mathsf{d}_K}{p}\pmb{\Big)}=0$, $\pmb{\Big(}\frac{\mathsf{d}_K}{p}\pmb{\Big)}=1$, respectively. We will use well-known properties of the Kronecker symbol (like the quadratic reciprocity law) throughout this note without further mention.

\begin{definition}[The conditions (C) and (SC)]\label{Definition 1.4} Recall how $x$ and $y$ were defined above. We say that {\it $d$ induces a counterexample to the \textnormal{(EMW)}-conjecture} (or {\it $d$ satisfies \textnormal{(C)}} for short) if $d\equiv 7\mod 8$ and there are some $k,u,v\in\mathbb{N}$ such that $u$ is powerful, $k$ and $v$ are odd, $\varepsilon^k=u+v\sqrt{d}$ and $d\mid v$. Furthermore, we say that {\it $d$ induces a strong counterexample to the \textnormal{(EMW)}-conjecture} (or {\it $d$ satisfies \textnormal{(SC)}} for short) if $d\equiv 7\mod 8$, $x$ is powerful, $y$ is odd and $d\mid y$.
\end{definition}

Clearly, if $d$ satisfies (SC), then $d$ satisfies (C). It is shown in \cite{Mo-Wa86} that the (EMW)-conjecture holds if and only if there is no squarefree $d\in\mathbb{N}_{\geq 2}$ that satisfies (C). This result of R. A. Mollin and P. G. Walsh provides us with a relationship between the (EMW)-conjecture and the squarefree integers $d\in\mathbb{N}_{\geq 2}$ with $d\mid y$. In Section~\ref{3}, we try to specify whether any of the numbers in the list~(\ref{List 1.1}) satisfies (C) or (SC). We prove that none of these numbers satisfies (SC) and that all but two do not satisfy (C). Nevertheless, we were unable to determine whether the remaining (two) numbers satisfy (C). For more details on the difficulties that arise here, we refer to \cite[p. 126]{Mo87}.

\begin{definition}[Conductors of relative class number one and the condition (RC)]\label{Definition 1.5} The integer $d$ is said to {\it have no nontrivial conductors of relative class number one} (or {\it to satisfy \textnormal{(RC)}} for short) if $\{f\in\mathbb{N}:{\rm h}(d)=|{\rm Pic}(\mathcal{O}_f)|\}=\{1\}$.
\end{definition}

The first systematic study (of which we are aware) of this condition was done in \cite{Fu-Pa12}. Following this, the problem of describing (RC) gained more traction \cite{Mo13} and was finally solved in \cite{Ch-Sa14}. We present the connection of (RC) and squarefree integers $d\in\mathbb{N}_{\geq 2}$ with $d\mid y$ in Proposition~\ref{Proposition 2.2}.

\begin{definition}[Unusual orders in real quadratic number fields]\label{Definition 1.6} Let $f\in\mathbb{N}$. We say that $f$ is an {\it unusual conductor of $d$} if $f$ is squarefree, $f$ is divisible by a ramified prime, $f$ is not divisible by a split prime, ${\rm h}(d)=|{\rm Pic}(\mathcal{O}_f)|=2$ and for each ramified $p\in\mathbb{P}$ with $p\mid f$ and all $a,b\in\mathbb{Z}$ we have $|pa^2-\frac{\mathsf{d}_K}{p}b^2|\not=4$. Let $D_d$ be the set of unusual conductors of $d$.
\end{definition}

The definition of an unusual conductor seems artificial, but becomes clear in view of the results of \cite{Br-Ge-Re20,Re23} (since these results provide a link to an important property in factorization theory). We discuss the relationships of unusual orders and squarefree integers $d\in\mathbb{N}_{\geq 2}$ with $d\mid y$ in Proposition~\ref{Proposition 2.4} and Theorem~\ref{Theorem 2.5} below.

\section{Results}\label{2}

We start with a lemma that will be useful in the subsequent discussion of the conditions (C) and (SC) (that were introduced in Definition~\ref{Definition 1.4} above). It will be applied in Section~\ref{3}, where we show that none of the $21$ members of the list~(\ref{List 1.1}) satisfies (SC) and all but (possibly) two of these members do not satisfy (C).

\begin{lemma}\label{Lemma 2.1}
Let $d$ satisfy \textnormal{(C)}. Then $y$ is odd.
\end{lemma}

\begin{proof}
By definition of (C), we deduce that $d\equiv 7\mod 8$ and that there are some $k,u,v\in\mathbb{N}$ such that $k$ and $v$ are odd and $\varepsilon^k=u+v\sqrt{d}$. Since $d\equiv 7\mod 8$, we have $x^2-dy^2={\rm N}(\varepsilon)=1$, and hence $xy$ is even. Therefore, $v=\sum_{i=0,i\equiv 1\mod 2}^k\binom{k}{i}x^{k-i}y^id^{\frac{i-1}{2}}\equiv y^kd^{\frac{k-1}{2}}\equiv y\mod 2$, and thus $y$ is odd.
\end{proof}

Our next result is a variant of the main theorem of \cite{Ch-Sa14}. It establishes a connection between the condition (RC) and the divisibility of $y$ by $d$.

\begin{proposition}\label{Proposition 2.2}
The number $d$ satisfies \textnormal{(RC)} if and only if ${\rm N}(\varepsilon)=1$, $d\not\equiv 1\mod 8$, $y$ is even and $d\mid y$.
\end{proposition}

\begin{proof}
First we recall some notation from \cite{Ch-Sa14}. Clearly, there exist unique $\alpha_0,\beta_0\in\mathbb{Q}$ such that $\varepsilon=\alpha_0+\beta_0\sqrt{d}$. Note that $2\alpha_0,2\beta_0\in\mathbb{N}$. We set

\[
\tilde{\beta}_0=\begin{cases}\beta_0&\textnormal{if }\varepsilon\in\mathbb{Z}[\sqrt{d}],\\ 2\beta_0&\textnormal{if }\varepsilon\not\in\mathbb{Z}[\sqrt{d}].\end{cases}
\]

Observe that $\tilde{\beta}_0\in\mathbb{N}$ and

\[
y=\begin{cases}\tilde{\beta}_0&\textnormal{if }d\not\equiv 1\mod 4\textnormal{ or }\varepsilon\not\in\mathbb{Z}[\sqrt{d}],\\ 2\tilde{\beta}_0&\textnormal{if }d\equiv 1\mod 4\textnormal{ and }\varepsilon\in\mathbb{Z}[\sqrt{d}].\end{cases}
\]

In particular, $\tilde{\beta}_0\mid y$ and $y\mid 2\tilde{\beta}_0$.

\smallskip
Now let $d$ satisfy (RC). We infer by \cite[Proposition 3.4]{Ch-Sa14} that ${\rm N}(\varepsilon)=1$. Moreover, it follows from \cite[Theorem 4.1]{Ch-Sa14} that $d\mid\tilde{\beta}_0$, and hence $d\mid y$. If $d$ is even, then clearly $d\not\equiv 1\mod 8$ and $y$ is even (since $d\mid y)$. Now let $d$ be odd. Then \cite[Theorem 4.1]{Ch-Sa14} implies that $d\not\equiv 1\mod 8$ and $\tilde{\beta}_0$ is even. Therefore, $y$ is even.

\smallskip
Conversely, let ${\rm N}(\varepsilon)=1$, let $d\not\equiv 1\mod 8$, let $y$ be even and let $d\mid y$. Next we show that $d\mid\tilde{\beta}_0$ and $\tilde{\beta}_0$ is even. Without restriction, we can assume that $d\equiv 1\mod 4$ and $\varepsilon\in\mathbb{Z}[\sqrt{d}]$. Since $d$ is odd, we know from $d\mid y=2\tilde{\beta}_0$ that $d\mid\tilde{\beta}_0$. Also note that $\alpha_0,\beta_0\in\mathbb{N}$ and $\beta_0=\tilde{\beta}_0$. Consequently, $1={\rm N}(\varepsilon)=\alpha_0^2-d\tilde{\beta}_0^2$, and thus $\alpha_0^2\equiv 1+\tilde{\beta}_0^2\mod 4$. If $\tilde{\beta}_0$ is odd, then $\alpha_0^2\equiv 2\mod 4$, a contradiction. This implies that $\tilde{\beta}_0$ is even. It is now an immediate consequence of \cite[Theorem 4.1]{Ch-Sa14} that $d$ satisfies (RC).
\end{proof}

\begin{lemma}\label{Lemma 2.3}
Let $p,q\in\mathbb{P}$ be such that $p\equiv 1\mod 4$, $q\equiv 3\mod 4$ and $d=pq$. If $y$ is even, then there are some $a,b\in\mathbb{Z}$ such that $|pa^2-qb^2|=1$. If $y$ is odd, then there are some $a,b\in\mathbb{Z}$ such that $|a^2-db^2|=2$ or there are some $a,b\in\mathbb{Z}$ such that $|pa^2-qb^2|=2$.
\end{lemma}

\begin{proof}
This is well-known and can be shown by investigating the norm of the fundamental unit. A detailed proof can be found in \cite[proof of Theorem 4.4, Case 3]{Re23}.
\end{proof}

In \cite[Theorem 5.4]{Re23} it was shown that the set of real quadratic number fields that have an order with an unusual conductor can (naturally) be divided into 41 disjoint subsets. It was also proved in \cite{Re23} that all but (possibly) one of these subsets are nonempty. The squarefree integers that define the real quadratic number fields in the aforementioned exceptional subset are called the {\it squarefree integers of type 4/form 1} (in the terminology of \cite[p. 88]{Re23}). Note that the squarefree integers $d$ that satisfy the conditions in Proposition~\ref{Proposition 2.4} below are precisely the squarefree integers $d$ of type 4/form 1. The hitherto open problem of their existence was the driving factor for the search conducted in \cite{Re23}. Recall that $D_d$ denotes the set of unusual conductors of $d$ (see Definition~\ref{Definition 1.6}) and ${\rm h}(d)$ denotes the class number of $K$.

\begin{proposition}\label{Proposition 2.4}
Let $p,q\in\mathbb{P}$ be such that $p\equiv 1\mod 4$, $q\equiv 3\mod 4$, $d=pq$ and ${\rm h}(d)=2$. The following conditions are equivalent:
\begin{enumerate}
\item[(1)] $D_d=\{2\}$.
\item[(2)] $p\equiv 5\mod 8$, $y$ is odd and $d\mid y$.
\item[(3)] $p\equiv 5\mod 8$, $\pmb{\Big(}\frac{p}{q}\pmb{\Big)}=-1$ and $d\mid y$.
\end{enumerate}
\end{proposition}

\begin{proof}
First, we show that if $p\equiv 5\mod 8$, then $y$ is odd if and only if $\pmb{\Big(}\frac{p}{q}\pmb{\Big)}=-1$. Let $p\equiv 5\mod 8$.

\smallskip
Let $y$ be odd. If there are some $a,b\in\mathbb{Z}$ such that $|a^2-db^2|=2$, then $\pmb{\Big(}\frac{2}{p}\pmb{\Big)}=1$, which contradicts the fact that $p\equiv 5\mod 8$. We infer by Lemma~\ref{Lemma 2.3} that there are some $a,b\in\mathbb{Z}$ such that $|pa^2-qb^2|=2$. Consequently, $\pmb{\Big(}\frac{p}{q}\pmb{\Big)}=\pmb{\Big(}\frac{q}{p}\pmb{\Big)}=\pmb{\Big(}\frac{2}{p}\pmb{\Big)}=-1$.

\smallskip
Now let $y$ be even. By Lemma~\ref{Lemma 2.3}, there are some $a,b\in\mathbb{Z}$ such that $|pa^2-qb^2|=1$. This implies that $\pmb{\Big(}\frac{p}{q}\pmb{\Big)}=\pmb{\Big(}\frac{q}{p}\pmb{\Big)}=1$.

\medskip
(1) $\Rightarrow$ (2) Since $2\in D_d$, it follows from \cite[Theorem 4.4]{Re23} that ${\rm h}(d)=|{\rm Pic}(\mathcal{O}_2)|$ and $\pmb{\Big(}\frac{2}{p}\pmb{\Big)}=-1$. Therefore, $p\equiv 5\mod 8$ and $y$ is odd by \cite[Theorem 5.9.7.4]{HK13a}. Since $\pmb{\Big(}\frac{p}{q}\pmb{\Big)}=-1$, it follows that $\pmb{\Big(}\frac{\alpha d/p}{p}\pmb{\Big)}=\pmb{\Big(}\frac{-\alpha q}{p}\pmb{\Big)}=-1$ for each $\alpha\in\{-1,1\}$, and since $p,q\not\in D_d$, we infer by \cite[Theorem 4.4]{Re23} that ${\rm h}(d)\not=|{\rm Pic}(\mathcal{O}_r)|$ for each $r\in\{p,q\}$. Therefore, $r\mid y$ for each $r\in\{p,q\}$ by \cite[Theorem 5.9.7.4]{HK13a}, and thus $d\mid y$.

\medskip
(2) $\Rightarrow$ (3) This is clear.

\medskip
(3) $\Rightarrow$ (1) Since $y$ is odd and $d\mid y$, we infer by \cite[Theorem 5.9.7.4]{HK13a} that ${\rm h}(d)=|{\rm Pic}(\mathcal{O}_2)|$ and $|{\rm Pic}(\mathcal{O}_p)|\not={\rm h}(d)\not=|{\rm Pic}(\mathcal{O}_q)|$. Since $p\equiv 5\mod 8$, it follows from \cite[Theorem 4.4]{Re23} that $2\in D_d$ and $p,q\not\in D_d$. Therefore, $D_d=\{2\}$ by \cite[Theorem 5.4]{Re23}.
\end{proof}

Finally, we present the main result of this note. It was the main motivation (besides Proposition~\ref{Proposition 2.4}) for the computer search discussed below.

\begin{theorem}\label{Theorem 2.5}
Let ${\rm h}(d)=2$ and suppose one of the following conditions is satisfied:
\begin{enumerate}
\item[(a)] There are distinct $p,q\in\mathbb{P}$ such that $p\equiv q\equiv 1\mod 4$, $d=pq$ and ${\rm N}(\varepsilon)=-1$.
\item[(b)] There are $p,q\in\mathbb{P}$ such that $p\equiv 1\mod 8$, $q\equiv 3\mod 4$, $d=pq$ and $y$ is odd.
\item[(c)] There are distinct $p,q\in\mathbb{P}$ such that $p\equiv q\equiv 3\mod 8$ and $d=2pq$.
\item[(d)] There are $p,q\in\mathbb{P}$ such that $p\equiv 1\mod 8$, $q\equiv 3\mod 4$, $\pmb{\Big(}\frac{p}{q}\pmb{\Big)}=-1$ and $d=2pq$.
\end{enumerate}
Then $D_d=\emptyset$ if and only if $d\mid y$.
\end{theorem}

\begin{proof}
It is a simple consequence of \cite[Theorem 5.9.7.4]{HK13a} that for each ramified $r\in\mathbb{P}$, ${\rm h}(d)\not=|{\rm Pic}(\mathcal{O}_r)|$ if and only if $r\mid y$. In what follows, we use this fact without further mention.

\smallskip
(a) Obviously, $\{p,q\}$ is the set of ramified primes. It follows immediately from \cite[Corollary 3.10(2)]{Re23} that $D_d=\emptyset$ if and only if ${\rm h}(d)\not=|{\rm Pic}(\mathcal{O}_r)|$ for each $r\in\{p,q\}$ if and only if $r\mid y$ for each $r\in\{p,q\}$ if and only if $d\mid y$.

\smallskip
(b) Clearly, $\{2,p,q\}$ is the set of ramified primes. Since $p\equiv 1\mod 8$, we have $\pmb{\Big(}\frac{2}{p}\pmb{\Big)}=1$, and hence $2\not\in D_d$ by \cite[Theorem 4.4]{Re23}. Moreover, $\pmb{\Big(}\frac{p}{q}\pmb{\Big)}=-1$ by \cite[Lemma 4.3]{Re23}, and thus for all $a,b\in\mathbb{Z}$, $|pa^2-qb^2|\not=1$. This implies that for each $r\in\{p,q\}$ and all $a,b\in\mathbb{Z}$, $|ra^2-\frac{\mathsf{d}_K}{r}b^2|\not=4$. We infer by \cite[Corollary 3.10(1)]{Re23} that $D_d=\emptyset$ if and only if ${\rm h}(d)\not=|{\rm Pic}(\mathcal{O}_r)|$ for each $r\in\{p,q\}$ if and only if $r\mid y$ for each $r\in\{p,q\}$ if and only if $d\mid y$.

\smallskip
(c) Observe that $\{2,p,q\}$ is the set of ramified primes and $x^2-dy^2=1$. Therefore, $y$ is even, and hence $2\not\in D_d$ by \cite[Theorem 4.4]{Re23}. Let $\alpha\in\{-1,1\}$. Then $\pmb{\Big(}\frac{\alpha d/p}{p}\pmb{\Big)}=\pmb{\Big(}\frac{2\alpha}{p}\pmb{\Big)}\pmb{\Big(}\frac{q}{p}\pmb{\Big)}=-\alpha\pmb{\Big(}\frac{q}{p}\pmb{\Big)}\not=-\alpha\pmb{\Big(}\frac{p}{q}\pmb{\Big)}=\pmb{\Big(}\frac{-\alpha p}{q}\pmb{\Big)}$, and hence $\pmb{\Big(}\frac{\alpha d/p}{p}\pmb{\Big)}=-1$ or $\pmb{\Big(}\frac{-\alpha p}{q}\pmb{\Big)}=-1$. It follows by analogy that $\pmb{\Big(}\frac{\alpha d/q}{q}\pmb{\Big)}=-1$ or $\pmb{\Big(}\frac{-\alpha q}{p}\pmb{\Big)}=-1$. We infer by \cite[Theorem 4.4]{Re23} that for each $r\in\{p,q\}$, $r\not\in D_d$ if and only if $r\mid y$. Since $y$ is even and $2\not\in D_d$, we deduce by \cite[Theorem 2.6(3)]{Re23} that $D_d=\emptyset$ if and only if $r\not\in D_d$ for each $r\in\{p,q\}$, if and only if $r\mid y$ for each $r\in\{p,q\}$, if and only if $d\mid y$.

\smallskip
(d) Here again $\{2,p,q\}$ is the set of ramified primes and $x^2-dy^2=1$. We infer that $y$ is even, and thus $2\not\in D_d$ by \cite[Theorem 4.4]{Re23}. Let $\alpha\in\{-1,1\}$. Observe that $\pmb{\Big(}\frac{\alpha d/p}{p}\pmb{\Big)}=\pmb{\Big(}\frac{2q}{p}\pmb{\Big)}=\pmb{\Big(}\frac{p}{q}\pmb{\Big)}=-1$ and $\pmb{\Big(}\frac{-\alpha q}{p}\pmb{\Big)}=\pmb{\Big(}\frac{q}{p}\pmb{\Big)}=-1$. Then \cite[Theorem 4.4]{Re23} implies that for each $r\in\{p,q\}$, $r\not\in D_d$ if and only if $r\mid y$. Since $y$ is even and $2\not\in D_d$, it follows from \cite[Theorem 2.6(3)]{Re23} that $D_d=\emptyset$ if and only if $r\not\in D_d$ for each $r\in\{p,q\}$, if and only if $r\mid y$ for each $r\in\{p,q\}$, if and only if $d\mid y$.
\end{proof}

\section{Examples and computational results}\label{3}

In what follows, let $X,Y\in\mathbb{N}_0$ be such that $X+Y\sqrt{d}$ is the fundamental unit of $\mathbb{Z}[\sqrt{d}]$ (i.e., $X+Y\sqrt{d}$ is the unique unit $\eta$ of $\mathbb{Z}[\sqrt{d}]$ such that $\eta>1$ and $\{\pm\eta^k:k\in\mathbb{Z}\}$ is the unit group of $\mathbb{Z}[\sqrt{d}]$). Observe that $X+Y\sqrt{d}\in\{\varepsilon,\varepsilon^3\}$ (see \cite[p. 372]{Ch-Sa14} or \cite[p. 621]{St-Wi88}). Let $\alpha\in\{0,1\}$ be such that $\alpha\equiv y\mod 2$ and let $\beta\in [0,7]$ be such that $\beta\equiv d\mod 8$. Moreover, let $s=|\{p\in\mathbb{P}:d\equiv 0\mod p\}|$ (i.e., $s$ is the number of distinct prime divisors of $d$). It follows from Proposition~\ref{Proposition 2.2} that $d$ satisfies (RC) if and only if $d\mid y$, $\alpha\not=1\not=\beta$ and ${\rm N}(\varepsilon)=1$. Obviously, if $d$ satisfies (C), then $\alpha=1$ (by Lemma~\ref{Lemma 2.1}).

\smallskip
Next we want to briefly discuss two algorithms to find squarefree $d\in\mathbb{N}_{\geq 2}$ with $d\mid y$. The first algorithm is called the {\it small step algorithm}. We use it to determine whether a squarefree integer $d\in\mathbb{N}_{\geq 2}$ satisfies $d\mid y$. The second algorithm is the {\it large step algorithm}. It is utilized to identify the squarefree integers $d\in\mathbb{N}_{\geq 1000000}$ with $d\mid Y$. It is well-known that if $d\mid y$, then $d\mid Y$. Moreover, if $d\mid Y$, then $d\mid 3y$. Also note that if $d\mid Y$ and $d\nmid y$, then $d\equiv 5\mod 8$, $3\mid d$ and $\varepsilon\not\in\mathbb{Z}[\sqrt{d}]$. An example ($d=17451248829$) of this behavior ($d\mid Y$ while $d\nmid y$) is given in \cite[above Remark 5.6]{Re23} and can also be found in the tables below. The large step algorithm is mainly used for search purposes (due to its better time complexity), while the small step algorithm is used for independent verification (and to handle the corner case with $d\mid Y$ and $d\nmid y$ that was mentioned before). For more details on the prior remarks and the algorithms used, we refer to \cite{St-Wi88}. Since the interval $[2,1.5\cdot 10^{12}]$ has already been searched \cite{Re23}, we now focus solely on the squarefree integers $d\geq 1.5\cdot 10^{12}$.

\smallskip
The main purpose of the following part is to present the results of our recent computer search. For this search, we used two implementations of the large step algorithm, a scalar implementation and a (partially) vectorized implementation with AVX-512. The vectorized version (with AVX-512) provides about 40\% more throughput than the scalar version on Zen 4 based CPUs. The programs were written in C and compiled with GCC-12.3.0 (with the compiler flag -O3). As a side note, we only used privately owned hardware for this computer search. We used 162 CPU cores (with hyperthreading and a clock rate around 4.1 GHz on average). Among these CPU cores are 74 cores with AVX-512 support (while the remaining 88 cores support AVX2). We did an exhaustive search on the squarefree integers $d\in [1.5\cdot 10^{12},5.325\cdot 10^{13}]$ (to find those that satisfy $d\mid Y$) and we spent approximately 3500 hours for this search in total.

\smallskip
Despite the fact that we performed an exhaustive search, we do not claim that the newly found numbers (four in total) are all the squarefree integers $d$ with $d\mid Y$ in the search interval. (It is likely that we found all of them.) The main reason is that we have currently not enough available computational resources for an independent double check (of all squarefree integers in the search interval).

\smallskip
Nevertheless, we tested each of the squarefree integers (in the tables) below with our (old and new) implementations of the small step algorithm and the large step algorithm. Furthermore, we used both Mathematica 12.0.0 and Pari/GP 2.15.2 to compute $\alpha,\beta,s,{\rm N}(\varepsilon)$ and ${\rm h}(d)$ in the tables below and to provide independent checks of the squarefree integers involved. Also note that our verifications with Mathematica and Pari/GP did not use the small step algorithm or the large step algorithm. These verifications were done by computing the fundamental unit of $\mathcal{O}_K$ (respectively $\mathbb{Z}[\sqrt{d}]$) in full, by extracting the component $y$ (respectively $Y$) and by using the ``mod operation'' to check whether $d\mid y$ (respectively $d\mid Y$).

\smallskip
It follows from Lemma~\ref{Lemma 2.1} that if $d$ is a squarefree integer of the tables below that satisfies (C), then $d\in\{4099215,39028039587479\}$. If $d=4099215$, then $d$ does not satisfy (SC), since $701\in\mathbb{P}$, $701\mid x$ and $701^2\nmid x$. Moreover, if $d=39028039587479$, then $d$ does not satisfy (SC), since $5\in\mathbb{P}$, $5\mid x$ and $5^2\nmid x$. In particular, none of the squarefree integers $d$ in the tables below satisfies (SC). We do not know if any $d\in\{4099215,39028039587479\}$ satisfies (C). In general, it is difficult to determine whether a specific squarefree $d\in\mathbb{N}_{\geq 2}$ with $d\equiv 7\mod 8$ satisfies (C). (To the best of our knowledge, it is even unknown if $d=7$ satisfies (C).) For more information, we refer to \cite[the paragraphs after Theorem 1 and Corollary 1, p. 126]{Mo87}.

\smallskip
Next we want to present the aforementioned counterexample (which can easily be derived from the tables below). We state it explicitly for the readers' convenience.

\begin{example}[The counterexample to Mordell's Pellian equation conjecture]\label{Example 3.1} Let $d=\linebreak 39028039587479$. Then $d\in\mathbb{P}$, $d\equiv 3\mod 4$ and $d\mid y$.
\end{example}

\begin{table}[htbp]
\centering
\begin{tabular}{|c|c|c|c|c|c|c|c|c|c|c|c|}
\hline
$d$ & $46$ & $430$ & $1817$ & $58254$ & $209991$ & $1752299$ & $3124318$ & $4099215$ & $5374184665$ & $6459560882$ & $16466394154$\\
\hline
$d\mid Y$ & \textnormal{true} & \textnormal{true} & \textnormal{true} & \textnormal{true} & \textnormal{true} & \textnormal{true} & \textnormal{true} & \textnormal{true} & \textnormal{true} & \textnormal{true} & \textnormal{true}\\
\hline
$d\mid y$ & \textnormal{true} & \textnormal{true} & \textnormal{true} & \textnormal{true} & \textnormal{true} & \textnormal{true} & \textnormal{true} & \textnormal{true} & \textnormal{true} & \textnormal{true} & \textnormal{true}\\
\hline
\textnormal{(RC)} & \textnormal{true} & \textnormal{true} & \textnormal{false} & \textnormal{true} & \textnormal{true} & \textnormal{true} & \textnormal{true} & \textnormal{false} & \textnormal{false} & \textnormal{true} & \textnormal{true}\\
\hline
$\alpha$ & $0$ & $0$ & $0$ & $0$ & $0$ & $0$ & $0$ & $1$ & $0$ & $0$ & $0$\\
\hline
$\beta$ & $6$ & $6$ & $1$ & $6$ & $7$ & $3$ & $6$ & $7$ & $1$ & $2$ & $2$\\
\hline
$s$ & $2$ & $3$ & $2$ & $5$ & $2$ & $3$ & $2$ & $3$ & $2$ & $4$ & $4$\\
\hline
${\rm N}(\varepsilon)$ & $1$ & $1$ & $1$ & $1$ & $1$ & $1$ & $1$ & $1$ & $-1$ & $1$ & $1$\\
\hline
${\rm h}(d)$ & $1$ & $2$ & $1$ & $8$ & $2$ & $4$ & $1$ & $4$ & $2$ & $4$ & $32$\\
\hline
\end{tabular}
\end{table}

\clearpage

\begin{table}[htbp]
\centering
\begin{tabular}{|c|c|c|c|c|c|c|}
\hline
$d$ & $17451248829$ & $20565608894$ & $25666082990$ & $117477414815$ & $125854178626$ & $1004569189366$\\
\hline
$d\mid Y$ & \textnormal{true} & \textnormal{true} & \textnormal{true} & \textnormal{true} & \textnormal{true} & \textnormal{true}\\
\hline
$d\mid y$ & \textnormal{false} & \textnormal{true} & \textnormal{true} & \textnormal{true} & \textnormal{true} & \textnormal{true}\\
\hline
\textnormal{(RC)} & \textnormal{false} & \textnormal{true} & \textnormal{true} & \textnormal{true} & \textnormal{true} & \textnormal{true}\\
\hline
$\alpha$ & $1$ & $0$ & $0$ & $0$ & $0$ & $0$\\
\hline
$\beta$ & $5$ & $6$ & $6$ & $7$ & $2$ & $6$\\
\hline
$s$ & $4$ & $3$ & $4$ & $4$ & $4$ & $2$\\
\hline
${\rm N}(\varepsilon)$ & $1$ & $1$ & $1$ & $1$ & $1$ & $1$\\
\hline
${\rm h}(d)$ & $4$ & $2$ & $8$ & $8$ & $8$ & $1$\\
\hline
\end{tabular}
\end{table}

\begin{table}[htbp]
\centering
\begin{tabular}{|c|c|c|c|c|c|}
\hline
$d$ & $1188580642033$ & $15826129757609$ & $18803675974841$ & $20256129307923$ & $39028039587479$\\
\hline
$d\mid Y$ & \textnormal{true} & \textnormal{true} & \textnormal{true} & \textnormal{true} & \textnormal{true}\\
\hline
$d\mid y$ & \textnormal{true} & \textnormal{true} & \textnormal{true} & \textnormal{true} & \textnormal{true}\\
\hline
\textnormal{(RC)} & \textnormal{false} & \textnormal{false} & \textnormal{false} & \textnormal{false} & \textnormal{false}\\
\hline
$\alpha$ & $0$ & $0$ & $0$ & $1$ & $1$\\
\hline
$\beta$ & $1$ & $1$ & $1$ & $3$ & $7$\\
\hline
$s$ & $3$ & $2$ & $3$ & $4$ & $1$\\
\hline
${\rm N}(\varepsilon)$ & $1$ & $1$ & $1$ & $1$ & $1$\\
\hline
${\rm h}(d)$ & $2$ & $1$ & $2$ & $16$ & $1$\\
\hline
\end{tabular}
\end{table}

It is likely that the example above is the smallest counterexample to Mordell's Pellian equation conjecture. (That said, we want to point out again that an independent double check of the search interval is still missing.) Furthermore, we want to emphasize that (to the best of our knowledge) the (AAC)-conjecture and the (EMW)-conjecture are still open. Besides that, it is also unknown (now, as before) whether squarefree integers of type 4/form 1 exist. As a final remark, we were able to write down a proof of Example~\ref{Example 3.1} that can be checked without computer assistance (in an acceptable amount of time). We intend to publish it in due time.

\bigskip
\noindent {\bf ACKNOWLEDGEMENTS.} We would like to thank A. Geroldinger for helpful suggestions and remarks. We also want to thank the anonymous referee for a large variety of corrections and comments that substantially improved the quality of this note.

\end{document}